\documentclass[12pt, reqno, a4paper, final]{amsart}

\usepackage{amsmath}                     % Matematiske kommandoer
\usepackage{a4}
\usepackage{appendix}
\usepackage{amsthm}
\usepackage{hyperref}
\usepackage{marvosym}                   % Se http://www.ctan.org/tex-archive/fonts/psfonts/marvosym/marvodoc.pdf
\usepackage{mathrsfs}									  % Ekstra kroellede bogstaver
\usepackage[all]{xy}										% Laekre kommutative diagrammer
\usepackage{url}												% %Inds�ttelse af www-links
\usepackage{verbatim,epsfig}
\usepackage[notref,notcite,draft]{showkeys} %viser hvad man har kaldt sine labels

\theoremstyle{plain}

\newtheorem{thm}{Theorem}[section] 
\newtheorem{theorem}{Theorem}[section] 

\newtheorem{corollary}[thm]{Corollary}

\newtheorem{lemma}[thm]{Lemma}

\newtheorem{proposition}[thm]{Proposition}
\newtheorem*{thm*}{Theorem}
\newtheorem*{prop*}{Proposition}
\newtheorem*{cor*}{Corollary}

\theoremstyle{definition}

\newtheorem{example}[thm]{Example}

\newtheorem{remark}[thm]{Remark}

\newcommand{\CC}{{\mathbb C}}

\newcommand{\GG}{{\mathbb G}}

\renewcommand{\L}{{\mathscr L}}

\newcommand{\ip}[2]{\langle {#1} \hspace{0.03cm} | \hspace{0.03cm} {#2} \rangle}

\newcommand{\varps}{{\varepsilon}}

\newcommand{\htens}{\bar{\otimes}}

\newcommand{\tens}{\otimes}

\renewcommand{\Re}{{\operatorname{Re}}}

\newcommand{\spann}{{\operatorname{span}}}

\newcommand{\To}{\longrightarrow}

\newcommand{\red}{{\operatorname{r}}}
\newcommand{\Tor}{\operatorname{Tor}}

\newcommand{\id}{\operatorname{id}}

\newcommand{\bet}{\beta^{(2)}}

\newcommand{\alg}{{\operatorname{alg}}}

\renewcommand{\int}{{\operatorname{int}}}

\newcommand{\Pol}{{\operatorname{Pol}}}

\makeatletter
\def\equalsfill{$\m@th\mathord=\mkern-7mu
\cleaders\hbox{$\!\mathord=\!$}\hfill
\mkern-7mu\mathord=$}
\makeatother   
\title{\texorpdfstring{On the zeroth $L^2$-homology of a quantum group}{On the zeroth L2-homology of a quantum group}}
\author{David Kyed} 
\address{David Kyed,
Mathematisches Institut,
Georg-Au\-gust-Uni\-versi\-t{\"a}t G{\"o}t\-ting\-en,
Bunsenstra{\ss}e 3-5,
D-37073 G{\"o}ttingen, 
Germany.}
\email{kyed@uni-math.gwdg.de}
\urladdr{www.uni-math.gwdg.de/kyed}



\thanks{Research supported by \emph{Deutsche Forschungsgemeinschaft} and 
 \emph{The Danish Council for In-}\\
 \indent\emph{dependent Research $|$ Natural Sciences}}

\begin{document}

\begin{abstract}We prove that the zeroth $L^2$-Betti number of a compact quantum group vanishes unless the underlying $C^*$-algebra is finite dimensional and that the zeroth $L^2$-homology itself is non-trivial exactly when the quantum group is coamenable.
\end{abstract}
\maketitle
 
\section{Introduction and Notation}
This note is an addendum to the results in \cite{coamenable-betti} and  \cite{quantum-betti}  concerning $L^2$-homology and $L^2$-Betti numbers for compact quantum groups. Although the necessary definitions will be given below, the reader not familiar with these notions might benefit from casting a sidelong glance at \cite{quantum-betti} while reading the present text.\\

Consider a compact quantum group $\GG$ in the sense of Woronowicz \cite{wor-cp-qgrps}; i.e.~$\GG$ consists of a (not necessarily commutative) unital, separable $C^*$-algebra $C(\GG)$ together with a unital $*$-homomorphism 
$\Delta_\GG\colon C(\GG)\to C(\GG)\tens C(\GG)$ which furthermore has to be coassociative and satisfy a certain non-de\-gene\-racy condition.  Recall that such a $C^*$-algebraic quantum group automatically gives rise to a purely algebraic quantum group (i.e.~a Hopf $*$-algebra \cite{klimyk}) whose underlying algebra will be denoted $\Pol(\GG)$, as well as a von Neumann algebraic quantum group \cite{kustermans-vaes} whose underlying algebra will be denoted $L^\infty(\GG)$. We also recall that the $C^*$-algebra $C(\GG)$  comes with a distinguished state $h_\GG$, called the \emph{Haar state}, which plays the role corresponding to the Haar measure on a genuine, compact group. Performing the GNS construction with respect to the Haar state yields a Hilbert space $L^2(\GG)$ on which $C(\GG)$ acts via the corresponding GNS-representation $\lambda$. The quantum group is said to be of \emph{Kac type} if its Haar state is a trace and to be \emph{finite} if its $C^*$-algebra $C(\GG)$ is finite dimensional. 

\begin{example}	
The fundamental example, on which the general definition is modeled, is obtained by considering a compact, second countable, Hausdorff topological group $G$ and its commutative $C^*$-algebra $C(G)$ of continuous, complex valued functions. In this case the comultiplication is the Gelfand dual of the multiplication map $G\times G\to G$ and the Haar state is given by integration against the unique Haar probability measure $\mu$ on $G$. The GNS-space therefore identifies with $L^2(G,\mu)$ and the representation $\lambda$ with the action of $C(G)$ on $L^2(G,\mu)$ by pointwise multiplication. Similarly, the von Neumann algebra identifies with $L^\infty(G,\mu)$ and the Hopf $*$-algebra is the subalgebra of $C(G)$ generated by matrix coefficients arising from irreducible, unitary representations of $G$. 
\end{example}

\begin{example}
Consider a countable, discrete group $\Gamma$. Denote by $C^*_\red(\Gamma)$ its reduced group $C^*$-algebra acting on $\ell^2(\Gamma)$ via the left regular representation and define a comultiplication on group elements by $\Delta_\red\gamma=\gamma\tens\gamma$. This turns $C^*_\red(\Gamma)$ into a compact quantum group whose Haar state is the natural trace on $C^*_\red(\Gamma)$. Hence the GNS-space and GNS-representation can be identified, respectively, with $\ell^2(\Gamma)$ and the left regular representation,  and the enveloping von Neumann algebra is therefore nothing but the group von Neumann algebra $\L(\Gamma)$. Each element in $\Gamma$ is a one-dimensional corepresentation for this quantum group and the Hopf $*$-algebra therefore identifies with the complex group algebra $\CC\Gamma$. 
\end{example}

To any  quantum group $\GG$ (compact as well as non-compact) a so-called multiplicative unitary $W$  on $L^2(\GG)\htens L^2(\GG)$ is associated; this is a unitary which (inter alia) has the property that
\begin{align*}
C(\GG_\red) \stackrel{\text{def}}{\hbox{\equalsfill}}\lambda(C(\GG))=\overline{\spann}_\CC\{(\id\tens \omega)W\mid \omega\in B(L^2(\GG))_*\}.
\end{align*}
Furthermore, a compact quantum group $\GG$ comes with  a dual quantum group $\hat{\GG}$ of so-called discrete type. This dual quantum group has 
\begin{align}\label{dual}
c_0(\hat{\GG})&=\overline{\spann}_\CC\{(\omega\tens \id)W\mid \omega\in B(L^2(\GG))_*\}
\end{align}
as underlying $C^*$-algebra and the multiplicative unitary $W$ naturally gives rise to a comultiplication on $c_0(\hat{\GG})$. For a more detailed treatment of $C^*$-algebraic (locally compact) quantum groups and their duality theory we refer to the work of Kustermans and Vaes \cite{kustermans-vaes-C*-lc}.\\

	In \cite{quantum-betti} the notion of $L^2$-homology and $L^2$-Betti numbers was introduced in the context of compact quantum groups of Kac type; if $\GG$ is such a quantum group its $L^2$-homology and $L^2$-Betti numbers are defined as
\[
H_n^{(2)}(\GG)=\Tor_n^{\Pol(\GG)}(L^\infty(\GG),\CC) \quad \text{ and } \quad \bet_n(\GG)=\dim_{L^\infty(\GG)}(H_n^{(2)}(\GG)).
\]
Here $\CC$ is considered  a $\Pol(\GG)$-module via the counit $\varps\colon \Pol(\GG)\to \CC$ and the dimension $\dim_{L^\infty(\GG)}(-)$ is L{\"u}ck's extended Murray-von Neumann dimension (calculated with respect to the tracial Haar state $h_\GG$) introduced in \cite{luck98}. As shown in \cite{quantum-betti} Proposition 1.3, this extends the classical  definition by means of  the formula $\bet_n(\GG)=\bet_n(\Gamma)$ when $C(\GG)=C^*_\red(\Gamma)$ for a discrete, countable group $\Gamma$.  \\

In this note we prove that the zeroth $L^2$-Betti number of a compact quantum group vanishes unless the underlying $C^*$-algebra is finite dimensional; this is done in Section \ref{beta-nul}. In Section \ref{Hnul} the zeroth $L^2$-homology is studied and we prove that it vanishes exactly when $\GG$ is non-coamenable.

\vspace{0.3cm}
\paragraph{\emph{Acknowledgements.}}
Theorem \ref{beta-north}  is an improvement of Proposition 2.2 in \cite{quantum-betti}. The author thanks Stefaan Vaes for pointing this improvement out to him. Furthermore, thanks are due to the anonymous referee for pointing out a simplification in the proof of Lemma \ref{surj-bij}.

\vspace{0.3cm}

\paragraph{\emph{Notation.}} Throughout the paper, the symbol $\odot$ will be used to denote algebraic tensor products while the symbol $\htens$ will be used to denote tensor products in the category of Hilbert spaces or in the category of von Neumann algebras. All tensor products between $C^*$-algebras are assumed minimal/spatial and these will be denoted by the symbol $\tens$. 

\section{\texorpdfstring{The zeroth $L^2$-Betti number}{The zeroth L2-Betti number}}\label{beta-nul}
In their fundamental paper \cite{CG}, Cheeger and Gromov proved that if $\Gamma$ is an infinite discrete group then $\bet_0(\Gamma)=0$. We prove here the following quantum group analogue of this result, improving \cite{quantum-betti} Proposition 2.2 by removing the factor assumption on the von Neumann algebra associated with the quantum group in question.
\begin{theorem}\label{beta-north}
If $\GG$ is a compact and infinite quantum group of Kac type then $\bet_0(\GG)=0$.
\end{theorem}
\begin{proof}
First note that
\[
H_0^{(2)}(\GG)=\Tor_0^{\Pol(\GG)}({L^\infty(\GG)},\CC)\simeq {L^\infty(\GG)}\underset{\Pol(\GG)}{\odot}\CC \simeq {L^\infty(\GG)}/J, 
\]
where $J$ denotes the left ideal in ${L^\infty(\GG)}$ generated by the kernel of the counit $\varps\colon \Pol(\GG)\to\CC$. Denote by $\bar{J}$ the strong operator closure of $J$ and note that
\[
J\subseteq \bar{J} \subseteq \overline{J}^{\alg}\stackrel{\text{def}}{\hbox{\equalsfill}}\bigcap_{\underset{J\subseteq \ker(f)}{f\in L^\infty(\GG)^*}}\hspace{-0.5cm}\ker(f),
\]
where $L^\infty(\GG)^*$ denotes the dual module; i.e.~the set 
\[
\{L^\infty(\GG)\ni x\mapsto xa\in L^\infty(\GG)\mid a\in L^\infty(\GG)\}.
\]
By \cite{luck98} Theorem 0.6 we have $\dim_{L^\infty(\GG)}(J)=\dim_{L^\infty(\GG)}(\overline{J}^\alg)$ and thus 
\[
\bet_0(\GG)=1-\dim_{L^\infty(\GG)}(J)=1-\dim_{L^\infty(\GG)}(\bar{J}). 
\]
We now aim to prove that $\bar{J}={L^\infty(\GG)}$ if $\GG$ is infinite. Assume, conversely, that that $\bar{J}\neq {L^\infty(\GG)}$ and note that $\bar{J}$ is also weak operator closed since $J$ is convex. Because $1\notin \bar{J}$, the counit $\varps\colon \Pol(\GG)\to \CC$  extends naturally to the weakly closed subspace 
\[
\CC +\bar{J}=\{\lambda 1+x\mid \lambda\in \CC, x\in \bar{J}\}\subseteq {L^\infty(\GG)},
\]
by setting $\varps(\lambda1 + x)=\lambda$.
%---------------------------------------------
%How to see that the sum is closed? Since $1\notin \bar{J}$ we can (by HB-Thm) choose a weakly cont. functional on $M$ with $\omega(1)=1$ and vanishing on $\bar{J}$. So, if $\lambda_i + a_i\To x$ then 
%\[
%\lambda_i=\omega(\lambda_i+a_i)\To \omega(x)
%\]
%Hence $(\lambda_i)$ is convergent and hence $(\lambda_i1)$ is convergent. This implies that $a_i=x-\lambda_i$ is strongly convergent - and hence that $x$ has the form $\lambda +a$.
%-----------------------------
To see that this extends $\varps$, just note that each element $a\in \Pol(\GG)$ can be written uniquely as the sum of a scalar and an element from $J$: $a=\varps(a)1+ (a-\varps(a)1)$. By \cite{KR1} Corollary 1.2.5, the extension $\varps\colon \CC+\bar{J}\to\CC$ is weakly continuous since its kernel $\bar{J}$ is weakly closed. The Hahn-Banach theorem therefore allows us to extend $\varps$ to a weakly continuous functional, also denoted $\varps$, on all of $B({L^2(\GG)})$. In particular, $\varps$ is weakly continuous on the unit ball of $B({L^2(\GG)})$ and thus $\varps\in B({L^2(\GG)})_*$. Denote by $\Lambda$ the natural inclusion $\Pol(\GG)\subseteq {L^2(\GG)}$ and by $W\in B({L^2(\GG)}\htens {L^2(\GG)})$ the multiplicative unitary for $\GG$, which for $x,y\in \Pol(\GG)$ is given by 
\[
W^*(\Lambda(x)\tens \Lambda(y))=(\Lambda\tens\Lambda)(\Delta_\GG(y)(x\tens 1)).
\]
For any $\omega\in B({L^2(\GG)})_*$ and any $x\in \Pol(\GG)$ we have
\begin{align}\label{snit-formel}
(\omega\tens \id)(W^*)(\Lambda(x))=\Lambda((\omega\tens \id)\Delta_\GG(x)).
\end{align}
This can be verified directly when $\omega$ has the form $T\mapsto \ip{T\Lambda(a)}{\Lambda(b)}$ and the general case follows from this since $B({L^2(\GG)})_*$ is the norm closure of the linear span of such functionals \cite[7.4.4]{KR2}. See e.g.~Result 1.2.5 in \cite{kustermans-vaes-C*-lc} for more details. Using the formula (\ref{snit-formel}) with $\omega=\varps$ we therefore obtain
\[
(\varps\tens\id)(W^*)=1.
\]
Since $\varps$ is weakly continuous,  $\varps\tens \id$ restricts to a $*$-homomorphism
\[
\varps\tens \id\colon {L^\infty(\GG)}\htens B({L^2(\GG)})\To B({L^2(\GG)}),
\]
and since $W\in {L^\infty(\GG)}\htens B({L^2(\GG)})$ it follows that $(\varps\tens\id)(W)=1$. This implies that the $C^*$-algebra $c_0(\hat{\GG})$ of the dual quantum group $\hat{\GG}$ (see e.g.~equation (\ref{dual}))
is unital and hence $\hat{\GG}$ is compact. Thus $\GG$ is both discrete and compact and $C(\GG)$ therefore  finite dimensional.
\end{proof}
\begin{remark}
If $C(\GG)$ has finite linear dimension $N$ it was proved in \cite{quantum-betti} Proposition 2.9
that the zeroth $L^2$-Betti number of $\GG$ equals $\frac{1}{N}$. By declaring $\frac{1}{\infty}=0$ we therefore have the formula
\[
\bet_0(\GG)=\dim_\CC(C(\GG))^{-1},
\]
for any compact quantum group $\GG$ of Kac type.
\end{remark}

In \cite{CS} Connes and Shlyakhtenko introduced $L^2$-homology and $L^2$-Betti numbers for certain tracial $*$-algebras. For these $L^2$-Betti numbers we get the following.

\begin{corollary}
If $\GG$ is an infinite, compact quantum group of Kac type  then the zeroth Connes-Shlyakhtenko $L^2$-Betti number of the tracial $*$-algebra $(\Pol(\GG),h_\GG)$ vanishes.
\end{corollary}
\begin{proof}
By Theorem 4.1 in \cite{quantum-betti} the Connes-Shlyakhtenko $L^2$-Betti numbers of the tracial $*$-algebra $(\Pol(\GG),h_\GG)$ are equal, degree by degree, to the $L^2$-Betti numbers of $\GG$.
\end{proof}

\section{\texorpdfstring{The zeroth $L^2$-homology}{The zeroth L2-homology}}\label{Hnul}

In this section we will focus on the zeroth $L^2$-homology module of a compact quantum group $\GG$.  Since the extended Murray-von Neumann dimension $\dim_{L^\infty(\GG)}(-)$ is not faithful, it may happen that the homology module $H_0^{(2)}(\GG)$ is non-trivial although, as we have just seen, its dimension $\bet_0(\GG)$  is zero whenever $\GG$ is infinite and of Kac type. In \cite{Luck02} Lemma 6.36, L{\"u}ck proves, for a discrete group $\Gamma$, that the zeroth $L^2$-homology $H_0^{(2)}(\Gamma)$ is non-vanishing exactly when $\Gamma$ is an amenable group and the aim of this section is to prove an analogue of this result for quantum groups. Since the $L^2$-homology modules $H_n^{(2)}(\GG)=\Tor_n^{\Pol(\GG)}(L^\infty(\GG),\CC)$ are defined also when the Haar state is not a trace\footnote{ the traciality is only needed in order for their dimension to be defined.}, we are going to consider the full class of compact quantum groups in this section.  Leaving the realm of quantum groups of Kac type gives rise to some minor technical problems since L{\"u}ck's results on finitely generated Hilbert modules are only available in the tracial setting, but since these technicalities are not of essential nature they are relegated to the appendix. The quantum group parallel  to L{\"u}ck's result takes the following form.

\begin{theorem}\label{H0-vanishing}
The zeroth $L^2$-homology of a compact quantum group $\GG$ is  non-vanishing if and only if  $\GG$ is  coamenable.
\end{theorem}
Recall that a compact quantum group $\GG$ is called coamenable \cite{murphy-tuset} if the counit $\varps\colon \Pol(\GG)\to\CC$ extends to a character on $C(\GG_\red)$. L{\"u}ck's proof for discrete groups \cite[3.36]{Luck02} is centered around Kesten's amenability condition  and since \cite{coamenable-betti} Theorem 4.4 (see also \cite{banica-subfactor}) provides us with a Kesten condition for quantum groups we can follow the same strategy here.

\begin{proof}
Denote by $(u^\alpha)_{\alpha\in I}$ a complete family of representatives for the equivalence classes of finite dimensional (not necessarily irreducible), unitary corepresentations of $\GG$. For a unitary corepresentation $u$ we denote by $d(u)$ its dimension (i.e.~matrix size) and by $\chi(u)$ its character  $\sum_{i=1}^{d(u)}u_{ii}$. Since
\[
\Pol(\GG)=\spann_{\CC}\{u_{ij}^\alpha\mid \alpha\in I,1\leq i,j\leq d(u^\alpha) \}
\]
we get
\[
\ker(\varps)=\spann_{\CC}\{ u_{ij}^\alpha-\varps(u_{ij}^\alpha)1\mid \alpha\in I,1\leq i,j\leq d(u^\alpha) \}.
\]
Letting $T$ denote the map
\begin{align}\label{T-defi}
\bigoplus_{\alpha\in I} \bigoplus_{i,j=1}^{d(u^\alpha)} \Pol(\GG)\ni (x_{ij}^\alpha)\longmapsto \sum_{\alpha\in I}\sum_{i,j=1}^{d(u^\alpha)} (u_{ij}^\alpha-\varps(u_{ij}^\alpha)1)x_{ij}^\alpha\in \Pol(\GG),
\end{align}
we therefore get an exact sequence of right ${\Pol(\GG)}$-modules
\begin{align*}
\bigoplus_{\alpha\in I} \bigoplus_{i,j=1}^{d(u^\alpha)} \Pol(\GG)\overset{T}{\To} \Pol(\GG)\overset{\varps}{\To}\CC\To 0.
\end{align*}
Applying the right exact functor $-\odot_{\Pol(\GG)}{L^\infty(\GG)}$ we obtain another exact sequence
\[
\bigoplus_{\alpha\in I} \bigoplus_{i,j=1}^{d(u^\alpha)} {L^\infty(\GG)}\overset{T}{\To} {L^\infty(\GG)}{\To}\CC\underset{\Pol(\GG)}{\odot}{L^\infty(\GG)}\To 0.
\]
We also denote the induced map by $T$ since it is given by the exact same formula just defined on a bigger domain. Recall that
\[
H_0^{(2)}(\GG)=\Tor_0^{\Pol(\GG)}(L^\infty(\GG),\CC)\simeq L^\infty(\GG)\underset{\Pol(\GG)}{\odot}\CC
\]
an thus $H_0^{(2)}(\GG)$ vanishes iff $T$ is surjective. We therefore have to prove that $T$ is surjective exactly when $\GG$ is non-coamenable.
For each $\alpha\in I$ we consider the restricted map
\[
\bigoplus_{i,j=1}^{d(u^\alpha)}L^\infty(\GG) \ni (x_{ij})\overset{T_\alpha}{\longmapsto} \sum_{i,j=1}^{d(u^\alpha)} (u_{ij}^\alpha-\varps(u_{ij}^\alpha)1)x_{ij}^\alpha\in L^\infty(\GG)
\]
and  claim  that $T$ is surjective if and only if $T_\alpha$ is surjective for some $\alpha\in I$. Clearly surjectivity of one $T_\alpha$ implies surjectivity of $T$ so assume, conversely, that $T$ is surjective. Then there exists a finite subset $I_0\subseteq I$ and 
\[
(a_{ij}^{\alpha})\in \bigoplus_{\alpha\in I_0}\bigoplus_{i,j=1}^{d(u^\alpha)}L^{\infty}(\GG) \ \text{ such that } \
\sum_{\alpha\in I_0}\sum_{i,j=1}^{d(u^\alpha)}(u_{ij}^\alpha-\varps(u_{ij}^{\alpha})1)a_{ij}^{\alpha}=1.
\]
But the direct sum $\oplus_{\alpha\in I_0}u^\alpha$ is again a finite dimensional corepresentation and hence equivalent to $u^\beta$ for some $\beta\in I$;  it follows that  $T_\beta$ is surjective. \\

Assume first that $T$ is surjective and pick (according to the claim just proven) an $\alpha\in I$ such that $T_\alpha$ is surjective.
Then Proposition \ref{right-exact} shows that also  the continuous extension 
\[
\tilde{T}_{\alpha}: \bigoplus_{i,j=1}^{d(u^\alpha)}L^2(\GG) \To L^2(\GG)
\]
is  surjective which by Lemma \ref{surj-bij} is equivalent to bijectivity of $\tilde{T}_\alpha^{\hspace{0cm}}{\tilde{T}_{\alpha}}^*$. Using the fact that $u^{\alpha}$ is a unitary matrix and that $\varps(u_{ij}^\alpha)=\delta_{i,j}$ (see \cite{woronowicz-pseudo} Proposition 3.2), a direct calculation verifies that  $\tilde{T}_\alpha^{\hspace{0cm}}{\tilde{T}_{\alpha}}^*$ is the continuous extension of the map on $\Pol(\GG)$ given by left multiplication with the element
%adjoint is given by a->((u_{ij}^*-\varps(u_{ij}))a)_{ij} for a\in M\subseteq L^2(M); what is left is now a calculation.
\[
2\Big{(}d(u^\alpha)-\frac{1}{2}\sum_{i=1}^{d(u^\alpha)} u_{ii}^\alpha+u_{ii}^{\alpha*}\Big{)},
\]
proving that
\[
\tilde{T}_{\alpha}^{}\tilde{T}_{\alpha}^*= 2\Big{(}d(u^\alpha)-\frac{1}{2}\sum_{i=1}^{d(u^\alpha)} \lambda(u_{ii}^\alpha)+\lambda (u_{ii}^\alpha)^*)\Big{)} =2\Big{(}d(u^\alpha)- \Re(\lambda(\chi(u^\alpha)))  \Big{)} .
\]
Thus $d(u^\alpha)$ is \emph{not} in the spectrum of the operator $\Re(\lambda(\chi(u^\alpha)))$ which by the Kesten condition \cite[4.4]{coamenable-betti} implies that $\GG$ can not be coamenable. If, conversely, $T$ is not surjective then $T_\alpha$ is non-surjective for each $\alpha\in I$ and  by the above analysis this means that $d(u^\alpha)$ is in the spectrum of $\Re(\lambda(\chi(u^{\alpha})))$ for each $\alpha\in I$. Since every finite dimensional, unitary corepresentation is equivalent to some $u^\alpha$ and since the characters of equivalent corepresentations are equal we conclude from the Kesten condition that $\GG$ is coamenable.

\end{proof}

\begin{remark}
Theorem \ref{H0-vanishing} is a very direct analogue of L{\"u}ck's result \cite[6.36]{Luck02}, but it also fits well with a result by Connes and Shlyakhtenko \cite[2.6]{CS} stating that the zeroth $L^2$-homology of a finite factor is non-vanishing exactly when the factor in question is the  hyperfinite (a.k.a.~amenable) one.
\end{remark} 

\begin{appendices}
\section{\texorpdfstring{Right exactness of $L^2$-completion}{Right exactness of L2-completion}}

We prove here a technical result  needed in the proof of Theorem \ref{H0-vanishing}. Let $M$ be a von Neumann algebra and let $\varphi$ be a faithful, normal state on $M$. Consider the GNS-space $H=L^2(M,\varphi)$ as well as  its associated GNS-representation $\pi\colon M\to B(H)$ and denote the natural inclusion $M\subseteq H$ by $\Lambda$.
The result needed is the following.
\begin{proposition}\label{right-exact}
A homomorphism $T\colon M^n\to M^m$ of finitely generated, free, right $M$-modules is surjective if and only if the continuous extension $\tilde{T}\colon H^n\to H^m$ is surjective.
\end{proposition}
\begin{remark}
In the case when $\varphi$ is a tracial state this follows directly from the fact that L{\"u}ck's $L^2$-completion functor  is exact with exact inverse \cite[6.24]{Luck02}.  
\end{remark}
Before giving the proof of Proposition  \ref{right-exact} we prove a small result of purely operator theoretic nature. The result is probably well known to operator algebraists, but since we were not able to find a reference we provide the proof for the convenience of the reader. 

\begin{lemma}\label{surj-bij}
Let $H$ and $K$ be Hilbert spaces and consider an operator $T$ in $B(H,K)$. Then $T$ is surjective if and only if $TT^*$ is bijective.
\end{lemma}
\begin{proof}
If $TT^*$ is bijective then clearly $T$ is surjective. If $T$ is surjective then $T^*$ is injective and by considering the polar decomposition $T^*=U(TT^*)^{\frac12}$ as well as the adjoint relation $T=(TT^*)^{\frac12}U^*$ we conclude that $(TT^*)^{\frac12}$ is both injective and surjective. Hence the same is true for $TT^*$.
%composition of 2 surjective maps is again surjective and injectivity follows since if $TT^*\xi=0 $ then $0=\ip{TT^*\xi}{\xi}=\| (TT^*)^{\frac12}\xi \|$.
\end{proof}

\begin{proof}[Proof of Proposition \ref{right-exact}]
Denote by $e_1,\dots, e_m$ the standard basis in $M^m$ and denote, for $\xi\in H$, by $\xi_i$ the vector in $H^m$ which has $\xi$ as its $i$-th coordinate and zeros everywhere else.  Since $T$ is right $M$-linear it is given by multiplication from the left by an $m\times n$ matrix $(a_{ij})$ with entries from $M$ and the extension $\tilde{T}$ is just the operator $(\pi(a_{ij}))\in B(H^n,H^m)$.  Assume first that $T$ is surjective. To prove that $\tilde{T}$ is surjective it is enough to show that $\xi_i$ is in its range for each $\xi\in H$ and $i\in\{1,\dots, m \}$. Since $T$ is surjective we can find $x_1,\dots, x_n\in M$ such that
\[
T\left[ {\begin{array}{c}
 x_1 \\
 \vdots  \\
 x_n
 \end{array} } \right] =
\left[ 
{\begin{array}{ccc}
 a_{11} & \cdots  &a_{1n}  \\
 \vdots &        &   \vdots\\
 a_{m1} & \dots& a_{mn}  \\
 \end{array} } 
 \right]
\left[ {\begin{array}{c}
 x_1 \\
 \vdots  \\
 x_n
 \end{array} } \right]= e_i,
 \]
and hence we get
\[
\tilde{T}
\left[ {\begin{array}{c}
 \pi(x_1)\xi \\
 \vdots  \\
 \pi(x_n)\xi
 \end{array} } \right] =\xi_i,
\]
and we conclude that $\tilde{T}$ is surjective. Assume, conversely, that $\tilde{T}$ is surjective. By Lemma \ref{surj-bij} the operator $\tilde{T}\tilde{T}^*$ is invertible and hence there exists an $m\times m$-matrix $S$ with entries in $M$ such that  $\tilde{T}\tilde{T}^*\tilde{S}=\id_{H^m}$.
To prove that $T$ is surjective we just need to prove that $e_1,\dots, e_m$ are all in its range. Fix $i\in \{1,\dots, m\}$ and note that $\tilde{T}\tilde{T}^*\tilde{S} \Lambda^m(e_i) =\Lambda^m (e_i)$.  Since both ${T}$ and ${S}$ are  matrices with entries from $M$ the vector $\tilde{T}^*\tilde{S} \Lambda^m(e_i)$ is of the form
\[
\left[ {\begin{array}{c}
 \Lambda(x_1) \\
 \vdots  \\
 \Lambda(x_n)
 \end{array} } \right] 
\]
for some  $x_1,\dots, x_n \in M$ and since $\Lambda$ is injective we conclude that 
\[
T
\left[ {\begin{array}{c}
 x_1 \\
 \vdots  \\
 x_n
 \end{array} } \right]=e_i.
\]

\end{proof}

\end{appendices}

\def\cprime{$'$}

\end{document}